\documentclass[12pt]{amsart}

\usepackage[dvips]{graphics}
\usepackage{epsfig}
\usepackage{amssymb}
\usepackage{amsthm}
\usepackage{amscd}
\usepackage[mathscr]{eucal}
\usepackage{enumitem}

\usepackage{hyperref}

\newtheorem{Theorem}{Theorem}[section]
\newtheorem{theorem}[Theorem]{Theorem}
\newtheorem{Proposition}[Theorem]{Proposition}

\newtheorem{lemma}[Theorem]{Lemma}
\newtheorem{Fact}[Theorem]{Fact}

\newtheorem{remark/def}[Theorem]{Remark/Definition}

\theoremstyle{definition}
\newtheorem{Example}[Theorem]{Example}

\newtheorem{Definition}[Theorem]{Definition}
\newtheorem{definition}[Theorem]{Definition}

\newtheorem{def/rem}[Theorem]{Definition/Remark}
\newtheorem{not/rem}[Theorem]{Notation/Remark}
\def \ind {\mathop{\smile \hskip -0.9em ^| \ }}
\def \dep {\mathop{ \not \smile \hskip -0.9em  ^| \ }}
\def \indk {\ind^K}
\def \depk {\dep^K}
\newcommand{\be}{\begin{enumerate}}
\newcommand{\bd}{\begin{defn}}

\newcommand{\bt}{\begin{theorem}}
\newcommand{\bl}{\begin{lemma}}
\newcommand{\ee}{\end{enumerate}}
\newcommand{\ed}{\end{defn}}
\newcommand{\et}{\end{theorem}}
\newcommand{\el}{\end{lemma}}

\newcommand{\la}{\langle}
\newcommand{\ra}{\rangle}

\newcommand{\CF}{{\mathcal F}}
\newcommand{\CG}{{\mathcal G}}
\newcommand{\CL}{{\mathcal L}}
\newcommand{\CH}{{\mathcal H}}

\newcommand{\CM}{{\mathcal M}}

\newcommand{\dom}{\mbox{dom}}

\newcommand{\NT}{\mathrm{NT}}

\def\dom{\operatorname{dom}}

\def\tp{\operatorname{tp}}

\title{More on  tree properties}

\author{Enrique Casanovas and  Byunghan Kim}
\date{June 14, 2019}

\address{Departament de Matem\`atiques i Inform\`atica\\
Universitat de Barcelona}
\address{Department of Mathematics\\
 Yonsei University}

\email{e.casanovas@ub.edu}

\email{bkim@yonsei.ac.kr}

\thanks{ The first author has been partially funded by a Spanish government grant MTM2017-86777-P and a Catalan  DURSI  grant 2017SGR-270. The second author has been supported by  Samsung Science Technology Foundation under Project Number SSTF-BA1301-03 and  an NRF of Korea grant 2018R1D1A1A02085584.}

\begin{document}
\maketitle
\begin{abstract}
Tree properties are introduced by Shelah, and it is well-known that 
a theory has TP (the tree property) if and only if it has
TP$_1$ or TP$_2$. 
In any simple theory (i.e., a theory not having TP), forking  supplies a good independence notion as it satisfies 
symmetry, full transitivity, extension, local character, and type-amalgamation, over sets. 
Shelah also introduced SOP$_n$ ($n$-strong order property).  Recently it is proved that  in any NSOP$_1$ theory (i.e. a 
theory not having SOP$_1$) holding  nonforking existence, Kim-forking also satisfies all the mentioned independence properties except base monotonicity (one direction of full transitivity). These results are the sources of motivation for this paper. 

Mainly, we produce type-counting criteria for SOP$_2$ (which is equivalent to TP$_1$) and  SOP$_1$. In addition, we study relationships between TP$_2$ and Kim-forking, and obtain that 
a theory  is supersimple iff there is no countably infinite  Kim-forking chain.  
\end{abstract}

In this paper we study various notions of tree properties, and we mainly produce type-counting criteria for SOP$_1$ and SOP$_2$. 
TP (the tree property) is introduced by S. Shelah in  \cite{S0}, and  it is shown that in any simple theory (a theory not having TP), 
forking satisfies
local character,  finite character, extension,  and later in  \cite{K},\cite{KP},
symmetry, full transitivity, and  type-amalgamation  of Lascar types, over arbitrary sets.

In \cite{S}, it is claimed that a theory has TP if and only if it has TP$_1$ or TP$_2$, and a complete proof is supplied in \cite{KKS}.
On the other hand, in \cite{S1}, Shelah  introduces the notions of $n$-strong order properties (SOP$_n$) for $n\geq 3$, which further classify  theories having 
TP$_1$. More precisely,  a theory has SOP$_n$ if there is a formula $\varphi(x,y)$ ($|x|=|y|$) defining a directed graph that has an infinite 
chain but no cycle of length $\leq n$. Hence SOP$_{n+1}$ implies SOP$_n$, but it is known that the implication is not reversible for each $n\geq 3$.
As we are not dealing with SOP$_n$ for $n\geq 3$ in this note, we do not give many details on this. 

\medskip

For $n=1,2,$ Shelah defines SOP$_n$ separately as follows.

 \begin{Definition}\be\item
We say a formula $\varphi(x,y)$ has {\em SOP$_2$} if
 there is a set $\{ a_\alpha\mid \alpha \in 2^{<\omega}\}$ of tuples such that 
 
 \be\item for each $\beta\in 2^\omega$, 
 $\{\varphi(x,a_{\beta\lceil n})\mid n\in\omega\}$
 is consistent, and 
 
 \item for each incomparable pair $\gamma,\gamma' \in  2^{<\omega}$,  
 $\{ \varphi(x, a_\gamma), \varphi(x, a_{\gamma'})\}$ is inconsistent.
\ee
A theory $T$ has SOP$_2$ if some formula in $T$ has SOP$_2$. 

\item
We say a formula $\varphi(x,y)$ has {\em SOP$_1$} if
 there is a set $\{ a_\alpha\mid \alpha \in 2^{<\omega}\}$ of tuples such that 
 \be\item for each $\beta\in 2^\omega$, 
 $\{\varphi(x,a_{\beta\lceil n})\mid n\in\omega\}$
 is consistent, and 
 
 \item for each $\beta \in  2^{<\omega}$, 
 $\{ \varphi(x, a_\gamma), \varphi(x, a_{\beta^\smallfrown 1})\}$ is inconsistent whenever $\beta^\smallfrown 0 \unlhd \gamma$. 
\ee
A theory $T$ {\em has SOP$_1$} if some  formula in $T$ has SOP$_1$. We say a theory $T$ {\em is NSOP$_1$} if $T$ does not have 
SOP$_1$. 
\ee
\end{Definition}

Hence it  follows that SOP$_2$ implies SOP$_1$. It is known for a theory that SOP$_3$ implies SOP$_2$, and SOP$_2$ is equivalent to TP$_1$. 
It is still an open question whether conversely, SOP$_1$ implies SOP$_3$, or SOP$_2$. The random parametrized equivalence relations (Example \ref{paraequi}), an infinite dimensional vector space over an algebraically closed field with a bilinear form, and 
 $\omega$-free PAC fields are typical examples having  non-simple but NSOP$_1$ theories. 
 Recently in \cite{KR},\cite{KR1},  it is shown that in any NSOP$_1$ theory, {\em over models},  `Kim-forking' satisfies all the aforementioned axioms that forking satisfies in simple theories, except base monotonicity (one direction of full 
 transitivity).  Then  it is proved in 
 \cite{DKR}, \cite{CKR} that the same axioms hold  {\em over arbitrary sets}  in any NSOP$_1$ theory having nonforking existence.    So far summarized  results justify our study of various tree properties in this paper.  
 
 \medskip 
 
 Throughout this note, we use standard notation. We work in a large saturated model $\CM$ of  a complete theory $T$ in a  language $\CL$, and $a,b,\dots$  ($A,B,\dots$) denote finite (small, resp.) tuples (sets, resp.) from $\CM$, unless said otherwise. 
 We write $a\equiv_A b$ to mean $\tp(a/A)=\tp(b/A)$. 
 As is customary, for cardinals $\kappa,\lambda$, we write $\lambda^\kappa$,  $\lambda^{<\kappa}$ to denote  $\{f\mid f: \kappa \to \lambda \}$, $\{f\mid f:\alpha \to \lambda, \ \alpha \in \kappa \}$ respectively,  or their  cardinalities, and it will be clear from context which one they mean. As usual, we can look  at $\lambda^{<\kappa}=\{f\mid f:\alpha \to \lambda, \ \alpha<\kappa \}$ as a tree, and we give a partial order $\unlhd$ to it. Namely we  let $\alpha\unlhd \beta$  for $\alpha, \beta\in \lambda^{<\kappa}$, when $\alpha=\beta\lceil\dom(\alpha)$. Thus we say $\alpha, \beta$ are {\em incomparable} if so are they in the ordering $\unlhd$. Also $\alpha^\smallfrown\beta$ denotes the concatenation of $\beta$ after $\alpha$.  When 
 $\beta=\la i_0,\dots, i_n\ra$ where $i_0,\dots,i_n\in \lambda$, we may simply write $\alpha i_0\cdots i_n$ to mean $\alpha^\smallfrown \beta$, so for example  $\alpha^\smallfrown 1$ or $\alpha 1$ indeed means $\alpha^\smallfrown\la 1\ra$. 
 In this note if we write a set as $\{p_i\mid i\in I\}$ then
 $p_i\ne p_j$ for $i\ne j\in I$. Given a sequence of tuples $\la c_i\mid i<\kappa\ra $ and $j<\kappa$, we write $c_{<j}$, $c_{>j}$ to abbreviate   $\la c_i\mid i<j\ra$, $\la c_i\mid j<i<\kappa \ra$, respectively.

\medskip

 We now state  definitions and facts including those already mentioned that will be freely used throughout the paper.

\begin{Definition} \be\item
We say an $\CL$-formula $\varphi(x,y)$ has the {\em $k$-tree property} ($k$-TP) where $k\geq 2$, if there is the set of tuples
$\{c_{\beta} \mid \beta\in \omega^{<\omega}\}$ (from $\CM$) such that
for each $\alpha\in \omega^{\omega}$, $\{\varphi(x,c_{\alpha\lceil
n})\mid n\in\omega\}$ is consistent, while for any $\beta\in \omega^{<\omega}$,
$\{\varphi(x,c_{\beta^\smallfrown i})\mid i\in\omega \}$ is
$k$-inconsistent (i.e. any $k$-subset is inconsistent). A formula has the {\em tree property} (TP) if it has $k$-TP for some $k\geq 2$. We say $T$ has TP if a formula in $T$ has this property. We say $T$ is {\em simple} if $T$ does not have TP. 

\item A formula $\psi(x,y)$ has the {\em tree property  of the first kind} (TP$_1$)
if there are
 tuples $a_{\alpha}$ ($\alpha\in \omega^{<\omega}$) such that   $\{\psi(x,a_{\beta\lceil
n})|n\in\omega\}$ is consistent for each $\beta\in \omega^{\omega}$,
while $\psi(x,a_{\alpha})\wedge
\psi(x,a_{\gamma})$ is inconsistent whenever  $\alpha,\gamma\in \omega^{<\omega}$
are incomparable. A theory has TP$_1$ if so has a formula.

\item We say  a formula $\psi(x,y)\in \CL$ has
 the {\em tree property of the second kind} (TP$_2$)
if there are tuples 
 $a^i_j$ ($i,j<\omega$) such that for each $i$,
 $\{ \psi(x,a^i_j) \mid j<\omega\}$ is 2-inconsistent, whereas for any $f\in \omega^{\omega}$,
  $\{ \psi(x,a^i_{f(i)}) \mid i<\omega\}$ is consistent. We say  $T$ has TP$_2$ if a formula has so in $T$.
\ee
\end{Definition}

 \begin{Fact} \label{sop2=tp1}
\be\item  The following are equivalent.
\be\item
 A theory $T$ has TP.
 \item $T$ has $2$-TP.
 \item $T$ has either TP$_1$ or TP$_2$.
 \ee

 \item
 A formula has TP$_1$ iff it has SOP$_2$.
 
 \item If a  formula has SOP$_1$ then it has $2$-TP.  
 \ee
 \end{Fact} 
 
 In Fact \ref{sop2=tp1}(1), the equivalence of (a) and (b)  is shown in \cite{S},
and that of (a) and (c)  is claimed in \cite{S}, but a correct proof is stated in \cite{KKS}.
  Fact \ref{sop2=tp1}(2)(3) easily come from the definitions.
 
\begin{Fact} \label{cr} \cite{CR} The following are equivalent. 
\be\item
A formula $\varphi(x,y)\in \CL$ has SOP$_1$.

\item There is a sequence $\la a_ib_i\mid i<\omega\ra$ such that 
\be\item $a_i\equiv_{(ab)_{<i}} b_i$ for all $i<\omega$, 

\item $\{\varphi(x,a_i)\mid i<\omega\}$ is consistent,  and 

\item  $\{\varphi(x,b_i)\mid i<\omega\}$ is $2$-inconsistent.
\ee\ee
\end{Fact}

 In Section 1, we  supply  type-counting criteria for SOP$_2$. These are  generalizations of  those in
 \cite{KK}, and  we use  similar techniques in \cite{C} where analogous  criteria  for TP are stated. 
 
 In Section 2, in  parallel, we produce type-counting criteria for SOP$_1$.

 In Section 3, we study TP$_2$ in relation with Kim-independence and local weights. In particular we show that  $T$ is supersimple iff there is no  Kim-forking chain of length $\omega$.

 \section{Type-counting criteria for SOP$_2$}

When Shelah introduces the class of simple theories in \cite{S0}, he states and proves type-counting criteria for TP.
Then in \cite{C}, the first author improves those and suggests more  elaborate  criteria for TP.
Later in  \cite{KK}, type-counting criteria for TP$_1$ (equivalently  for SOP$_2$) analogous to the type-counting results of \cite{S0} are suggested. 
In \cite{MS},  another type-counting criteria for SOP$_2$ is suggested. 
Now in this section, we supply more refined criteria for SOP$_2$, which are analogous to those for TP in \cite{C}.

 \begin{definition}  \label{nt2} Let $\varphi(x,y)$ be an $\CL$-formula.  Assume  infinite cardinals $\kappa,\lambda$ are given. We define  
 $\NT_\varphi^2(\kappa,\lambda)$ as the supremum of cardinalities $|\mathcal{F}|$ of sets $\mathcal{F}$  of positive $\varphi$-types $p(x)$  over some fixed set $A$ of cardinality $\lambda$  satisfying  that 
  \begin{enumerate}
\item  $|p(x)|=\kappa$ for every $p(x)\in\mathcal{F}$, and 
\item for every subfamily $\{p_i\mid i<\lambda^+\}\subseteq \mathcal{F}$,   there are  disjoint subsets $\tau_j \subset \lambda^+$  with $|\tau_j|=\lambda^+$, and families
$\{p'_i\mid p'_i\subseteq p_i,\ i\in \tau_j\}$ $(j=0,1)$ such that  $|p_i\smallsetminus p^\prime_i| <\kappa$ for each $i\in\tau_0\cup \tau_1$,  and  every formula 
in $\bigcup_{i\in\tau_1}p^\prime_i$  is inconsistent with every formula in 
$\bigcup_{i\in\tau_0}p^\prime_i$.
\end{enumerate}
Notice that if $|\CF|\leq \lambda$ then the condition (2) is vacuous. 

We define $\NT^2(\kappa,\lambda)$ in a similar way, with the only difference that each partial type $p(x)\in \mathcal{F}$ (with finite $x$) may contain any formula over $A$,  not only instances of a fixed $\varphi(x,y)$, while still $|p(x)|=\kappa$.
\end{definition}

Now given a formula $\varphi$, we give type-counting criteria for SOP$_2$, in terms of NT$^2_\varphi$.

\begin{Theorem}\label{sopt4}  Let  $\kappa,\lambda$ denote  infinite cardinals. The following are equivalent for a formula  $\varphi(x,y)\in \CL$.
\begin{enumerate}
\item $\varphi(x,y)$  has SOP$_2$.

\item $\NT_\varphi^2(\omega,\omega)\geq \omega_1$

\item $\NT_\varphi^2(\omega,\omega)\geq  2^\omega.$

\item $\NT_\varphi^2(\kappa,\lambda)\geq  \lambda^+$ for some $\kappa, \lambda$.

\item $\NT_\varphi^2(\kappa,\lambda)\geq \lambda^+$ for any 
$\kappa,\lambda$ with
$\lambda^{<\kappa}=\lambda$ and $\lambda^\kappa>\lambda$.

\item $\NT_\varphi^2(\kappa,\lambda)\geq  \lambda^\kappa$ for any 
$\kappa,\lambda$ such that 
$\lambda^{<\kappa}=\lambda$ and $\lambda^\kappa>\lambda$.
\end{enumerate}
\end{Theorem}
\begin{proof}
(1)$\Rightarrow$(6) \ 
Assume $\varphi(x,y)$ has SOP$_2$. Suppose that 
for infinite $\kappa, \lambda$, we have  $\lambda^{<\kappa}=\lambda$ and  $\lambda^\kappa>\lambda$. 
Hence $\kappa\leq \lambda$. We will show that $\NT^2_\varphi(\kappa, \lambda)\geq \lambda^\kappa$. 

Since $\varphi$ has TP$_1$ as in Fact \ref{sop2=tp1}(2),
by compactness, there is  a tree of formulas
  $\{\varphi(x,a_\sigma)| \ \sigma\in\lambda^{<\kappa}\}$ witnessing  TP$_1$ w.r.t. $\lambda^{<\kappa}$\, (i.e. for each $\beta\in \lambda^\kappa$,  $q_{\beta}(x): = \{ \varphi (x, a_{\beta\lceil i})\mid i<\kappa \}$ is consistent,  while for any incomparable $\alpha,\gamma \in \lambda^{<\kappa}$, 
  $\{ \varphi(x,a_\alpha),\varphi(x,a_{\gamma})\}$ is inconsistent). 
  Let $A$ be the set of
parameters in the tree.
We let $\CF: = \{ q_{\beta}(x)\mid \beta \in \lambda^\kappa \}$. Note that 
$|\CF|=\lambda^\kappa>
\lambda=\lambda^{<\kappa}=|A|.$ 

\medskip

We want to show that $\CF$ satisfies the condition (2) in  Definition \ref{nt2}.
Thus assume a set $\CG=\{q_\beta\mid \beta\in\tau \}$ is given, where 
$\tau\subseteq \lambda^\kappa$ with $|\tau|=\lambda^+$. Now 
for each $\sigma\in \lambda^{<\kappa}$, we let
$\CG_\sigma:=\{p\in \CG \mid \varphi(x,a_{\sigma})\in p\}$.

\medskip

\noindent {\em Claim.}  There are $\mu\in \lambda^{<\kappa}$ and $s_0< s_1\in \lambda$ such that
$|\CG_{\mu^\smallfrown \la s_0\ra}|=|\CG_{\mu^\smallfrown\la s_1\ra}|=\lambda^+$:     Suppose  not. Then for each 
$\sigma\in \lambda^{<\kappa}$ there is at most one $s<\lambda$ such that  
$|\CG_{\sigma^\smallfrown\la s\ra}|=\lambda^+$. Thus the only possibility is that there is $\delta \in \lambda^\kappa$ such that 
for each $i \in \kappa$, $|\CG_{\delta \lceil i}|=\lambda^+$, while for each $j\in \lambda$ with $j\ne \delta (i)$, we have $|\CG_{(\delta \lceil i)^\smallfrown \la j\ra }|\leq \lambda$. Since 
$$\CG=\{q_\delta\}\cup \bigcup \{ \CG_{(\delta \lceil i)^\smallfrown \la  j_i\ra}\mid i<\kappa, \ j_i<\lambda, \ j_i\ne \delta(i)\},$$  it follows that  $|\CG|\leq 1+\lambda\cdot \lambda\cdot \kappa=\lambda$, a contradiction.
Hence the claim follows.

\medskip

Now let   $\tau_0,\tau_1$ be
the disjoint  subsets of $\tau$ indexing  the sets $\CG_{\mu^\smallfrown \la s_0\ra}$ and $\CG_{\mu^\smallfrown\la s_1\ra}$, respectively,
so $\CG_{\mu^\smallfrown \la s_j\ra}=\{p_i\in \CG\mid i\in \tau_j\}$ ($j=0,1$).
We now put for each $i\in \tau_0\cup \tau_1$, $p'_i:=p_i\smallsetminus q_\mu$ where 
$q_\mu=\{\varphi(x,a_\sigma)| \ \sigma \trianglelefteq \mu \}$.
Hence $|p_i\smallsetminus p'_i|<\kappa$.
Moreover clearly each formula in $ \bigcup_{i\in\tau_1} p'_i$ is inconsistent with every formula 
in $\bigcup_{i\in \tau_0} p'_i$.
Therefore Definition \ref{nt2}(2) holds.

\medskip

(6)$\Rightarrow$(5)$\Rightarrow$(2)$\Rightarrow$(4)
and (6)$\Rightarrow$(3)$\Rightarrow$(2) Clear.

\medskip

(4)$\Rightarrow$(1)
Assume  $\NT_\varphi^2(\kappa,\lambda)\geq \lambda^+$ for some infinite $\kappa$ and $\lambda$. Hence there is a family $\CF=\{q_i\mid i<\lambda^+\}$ over
 a set $A$ with $|A|\leq \lambda$ satisfying the condition in Definition \ref{nt2}(2). 
We will produce an SOP$_2$ tree for $\varphi$ from $\CF$.

\medskip

\noindent {\em Claim.} There exist a  function $f: 2^{<\omega}\to A$, a family $\{\, \CG_\sigma\mid \sigma\in  2^{<\omega}\, \}$ of types, and a family 
$\, \{\, \tau_\sigma \subseteq
 \lambda^+\mid \sigma\in  2^{<\omega} \}\, $ such that  for all $\sigma\in  2^{<\omega}$, 

\be
\item[(i)] $|\tau_{\sigma}|=\lambda^+$;  $\tau_{\sigma 0}$ and $\tau_{\sigma 1 }$ are disjoint subsets of $\tau_{\sigma}$,

\item[(ii)]  $\CG_\sigma $ is of the form $\{p_i\mid p_i\subseteq q_i,\ i\in \tau_\sigma\}$ (so $|
\CG_\sigma|=\lambda^+$) with
$|q_i\smallsetminus p_i|<\kappa$, and for $j\in \{0,1\}$, $\CG_{\sigma j}$ is of the form $\{p'_i\mid p'_i\subseteq p_i\in \CG_\sigma, \ i\in \tau_{\sigma j}\}$ with $|p_i\smallsetminus p'_i|<\kappa$, 

\item[(iii)]  for $a_\sigma:=f(\sigma)$ we have    $\varphi(x, a_\sigma)\in \bigcap \CG_\sigma$, and

\item[(iv)] each formula in   $\bigcup \CG_{\sigma0}$ 
 is inconsistent with every formula in $\bigcup\CG_{\sigma1}$.
\ee

\medskip

\noindent {\em Proof of Claim.}
We construct such a function and sets by induction on the length of $\sigma$. When $\sigma=\emptyset$, choose $\varphi(x,b_i)$ from each $q_i\in\CF$. Then 
since $|A| < \lambda^+$ and $\lambda^+$ is regular (or just by counting), there must be   a subset
$\tau_{\emptyset}\subseteq \lambda^+$ of size $\lambda^+$ such that
$b_i$ are equal (say, to $a_{\emptyset}$) for all $i\in \tau_{\emptyset}$. Then set $f(\emptyset)=a_\emptyset$. Also, set $\CG_{\emptyset}:=  \{q_i \mid i\in \tau_{\emptyset}\}$, so $\varphi(x,a_\emptyset)\in \bigcap \CG_\emptyset$.

\medskip
Assume now the induction hypothesis for $\sigma$. We will find
sets and  function values corresponding to
$\sigma 0$ and $\sigma 1$. Write  $\CG_\sigma=\{p_i\mid i\in \tau_\sigma\}$.
Since $\CF$ satisfies Definition \ref{nt2}(2), there exist disjoint subsets $\tau'_{\sigma j } \subseteq \tau_{\sigma}\, $ of size $\lambda^+$ ($j=0,1$) and a subset $p'_i\subseteq q_i$ with $|q_i \smallsetminus p'_i|<\kappa$ for each $i\in \tau'_{\sigma 0 }\cup \tau'_{\sigma 1 }$, such that  every formula in $\bigcup_{i\in \tau'_{\sigma0}}p'_i$ is  inconsistent with each formula in $\bigcup_{i\in \tau'_{\sigma1}}p'_i$.  
We now let $p''_i:=p_i\cap p'_i$ for $i\in \bigcup_{j=0,1}\tau'_{\sigma j }$, and let 
$\CG'_{\sigma j}:=\{p''_i \mid i\in \tau'_{\sigma j}\}$. Then clearly $p''_i\subseteq p_i$,  $|q_i \smallsetminus p''_i|<\kappa$, and $|p_i \smallsetminus p''_i|<\kappa$.

 Now since again $|A|\leq \lambda$,  for $j\in \{0,1\}$, 
there must be a set $\tau_{\sigma j}\subseteq \tau'_{\sigma j}$ with $|\tau_{\sigma j}|=\lambda^+$  such that   for some $d_j\in A$ (which we put $a_{\sigma j}=f(\sigma j)$),   $\varphi(x,d_j)\in \bigcap_{i\in \tau_{\sigma j}} p''_i$. Therefore if we let 
$\CG_{\sigma j}:=\{ p''_i\mid i\in \tau_{\sigma j}\}$, then  $\tau_{\sigma j }$, $f(\sigma j)$ and $\CG_{\sigma j}$, for $j=0, 1$, satisfy all the required conditions for the induction step, and the proof for Claim is complete.

\medskip

Now, using the properties  described in Claim, we see that  the tree $ \{\varphi(x,a_{\sigma})\mid  \sigma\in 2^{<\omega}\}$ witnesses SOP$_2$.
 Indeed given any  $\sigma,\beta,\gamma \in 2^{<\omega}$,
the formula $\varphi(x,a_{\sigma^\smallfrown 0^\smallfrown \beta})$ is inconsistent with  $\varphi(x,a_{\sigma^\smallfrown 1^\smallfrown  \gamma})$.
\end{proof}

We now give type-counting criteria for SOP$_2$, for a theory. 

\begin{Theorem}\label{sopt5} Let $\kappa,\lambda$ denote infinite cardinals.  The following are equivalent.
\be

\item $T$  has SOP$_2$.

\item For every regular  $\kappa>|T|$, there is  $\lambda\geq 2^\kappa$ such that  
$\NT^2(\kappa,\lambda)> \lambda$.

\item  For some   regular $\kappa>|T|$ and some $\lambda\geq 2^\kappa$, 
we have $\NT^2(\kappa,\lambda)> \lambda$.

\item For every $\kappa,\lambda$ with   $\lambda^{<\kappa}= \lambda$ and $\lambda^\kappa>\lambda$, 
we have     $\NT^2 (\kappa,\lambda)\geq \lambda^\kappa$.

\item For every $\kappa,\lambda$ with   $\lambda^{<\kappa}= \lambda$ and $\lambda^\kappa>\lambda$, 
we have     $\NT^2 (\kappa,\lambda)> \lambda$.
\ee
\end{Theorem}
\begin{proof} 
(1)$\Rightarrow$(4) \ The same proof of (1)$\Rightarrow$(6) for Theorem \ref{sopt4} shows this.

 \medskip

(2)$\Rightarrow$(3),  (4)$\Rightarrow$(5)  \  Clear.

\medskip

(3)$\Rightarrow$(1) \  Assume (3) with the given $\kappa,\lambda$. Hence there is a family $\CF$ of arbitrary  types over $A$ with $|A|=\lambda$, 
satisfying  the conditions (2) and (3) in Definition \ref{nt2}. There is no harm to assume that  $|\CF|=\lambda^+$ and we write $\CF=\{q_i\mid i<\lambda^+\}.$ Since $|q_i|=\kappa$,
we write $q_i=\{\varphi^i_\alpha(x,a^i_\alpha)\mid \alpha <\kappa\},$ where $a^i_\alpha \in A$. Now since $|T|^\kappa=2^\kappa<\lambda^+$, there must be a subset 
$\tau$ of $\lambda^+$ with $|\tau|=\lambda^+$ such that the sequence 
$\la \varphi^i_\alpha(x,y^i_\alpha)\mid \alpha <\kappa\ra$ stays the same, say $\la \varphi_\alpha(x,y_\alpha)\mid \alpha <\kappa\ra$,  for every
$i\in \tau$. 
Moreover since  $\kappa(>|T|)$ is regular, there must be a subset $\mu\subseteq \kappa$ of size $\kappa$ such that  $\varphi_\alpha(x,y_\alpha)$ stays the same, say $\varphi(x,y)$, for all $\alpha\in\mu$.  Now we let 
$\CF_1:=\{\{\varphi(x, a^i_\alpha)\mid \alpha\in\mu\}\mid i\in \tau\}.$ Then it easily follows that $\CF_1$ also satisfies  Definition \ref{nt2}(1) and (2).
Moreover each type in $\CF_1$ is a positive $\varphi$-type. Therefore (1) follows
by Theorem \ref{sopt4}(4)$\Rightarrow$(1).

 \medskip

(5)$\Rightarrow$(2) \  Assume (5). Now given regular $\kappa> |T|$, let 
$\lambda:=\beth_\kappa(\kappa)$. Then $\lambda^{<\kappa}= \lambda <\lambda^\kappa$. Hence by (5), we have  $\NT^2(\kappa,\lambda)> \lambda$.
\end{proof}

\section{Type-counting criteria for SOP$_1$}

As said in the beginning of Section 1, type-counting criteria for  SOP$_2$ are given in \cite{KK} as well. 
But for the first time, here we  state and prove type-counting criteria for a formula to have SOP$_1$.

\begin{Definition}
We say a formula $\varphi(x,y)$ has {\em $\omega^{<\omega}$-SOP$_1$} if
 there is a set $\{ a_\alpha\mid \alpha \in \omega^{<\omega}\}$ of tuples such that 
 \be\item for each $\beta\in \omega^\omega$, 
 $\{\varphi(x,a_{\beta\lceil n})\mid n\in\omega\}$
 is consistent, and 
 
 \item for each $\beta \in  \omega^{<\omega}$ and each pair $m<n\in\omega$, 
 $\{ \varphi(x, a_\gamma), \varphi(x, a_{\beta n})\}$ is inconsistent whenever $\beta m\unlhd \gamma$. 
\ee
\end{Definition}

\begin{Fact}\label{omom}
A formula has SOP$_1$ iff it has $\omega^{<\omega}$-SOP$_1$.
\end{Fact}
\begin{proof}
$(\Leftarrow)$ Clear.

\medskip

$(\Rightarrow)$ Assume $\varphi(x,y)$ and 
$\{a_\alpha \mid \alpha \in 2^{<\omega}\}$ witness SOP$_1$. Now for each $n>1$, define a $1-1$ map $f_n: n^{<\omega}\to 2^{<\omega}$ such that  $f_n(\emptyset):=\emptyset$, and for $\alpha \in n^{<\omega}$ and $m<n$,  $f_n(\alpha m):= f_n(\alpha)\overbrace{0\cdots0}^{n-m-1}1$.

It follows that 
$A_n:=\{a_{f_n(\alpha)}\mid \alpha\in n^{<\omega} \}$ forms an 
$n^{<\omega}$-SOP$_1$ tree for $\varphi$, and then compactness yields an $\omega^{<\omega}$-SOP$_1$ tree for the formula.
\end{proof}

\begin{Definition}\label{nt1}
Let $\varphi(x,y)\in \CL$.  For any two infinite cardinals $\kappa,\lambda$, we define  $\NT_\varphi^1(\kappa,\lambda)$ as the supremum of cardinalities $|\mathcal{F}|$ of sets $\mathcal{F}$  of positive $\varphi$-types   
over some fixed set $A$ of cardinality $\lambda$  satisfying  that  
\be
\item  $|q(x)|= \kappa$ for every $q(x)\in\mathcal{F}$, and 
\item given any subfamily $\CG=\{q_i\mid i<\lambda^+\}$ of $\CF$ and a family $\CG'=
\{p_i\mid p_i\subseteq q_i, \ i<\lambda^+\}$ where $|q_i\smallsetminus p_i|<\kappa$
for each $i<\lambda^+$, 
there are disjoint subsets $\tau_0, \tau_1$ of $\lambda^+$
with $|\tau_j|=\lambda^+$ ($j=0,1$), and $\CG'_j=\{p'_i\mid p'_i\subseteq p_i, \ i\in\tau_j\}$ with $|p_i\smallsetminus p'_i|<\kappa$  for each $i\in \tau_0\cup \tau_1$,
such that for every $p'_i\in \CG'_1$ there is a formula  in $ p'_i$ which  is inconsistent with each formula in  $\bigcup \CG'_0$.
\ee
Notice that if $|\CF|\leq \lambda$ then the condition (2) is vacuous. 
\end{Definition}

\begin{Theorem}\label{positivephi}
Assume $\varphi(x,y)$ is an $\CL$-formula, and $\kappa,\lambda$ denote infinite cardinals.  The following are equivalent.
\be
\item $\varphi(x,y)$ has  SOP$_1$.

\item $\NT_\varphi^1(\omega,\omega)\geq \omega_1$

\item $\NT_\varphi^1(\omega,\omega)\geq  2^\omega.$

\item $\NT_\varphi^1(\kappa,\lambda)\geq  \lambda^+$ for some $\kappa, \lambda$.

\item $\NT_\varphi^1(\kappa,\lambda)\geq \lambda^+$ for any 
$\lambda$  and any regular $\kappa$ with
$\lambda^{<\kappa}=\lambda$ and $\lambda^\kappa>\lambda$.

\item $\NT_\varphi^1(\kappa,\lambda)\geq  \lambda^\kappa$ for any 
$\lambda$  and regular $\kappa$  such that 
$\lambda^{<\kappa}=\lambda$ and $\lambda^\kappa>\lambda$.
\ee
\end{Theorem}
\begin{proof} 
(1)$\Rightarrow$(6) \ 
Assume $\varphi(x,y)$ has SOP$_1$. Suppose that 
for regular  $\kappa$, and infinite $\lambda$, we have  $\lambda^{<\kappa}=\lambda$ and  $\lambda^\kappa>\lambda$. We will show that $\NT^1_\varphi(\kappa, \lambda)\geq \lambda^\kappa$. 

Since $\varphi$ has $\omega^{<\omega}$-SOP$_1$ as in Fact \ref{omom},
by compactness, there is  a tree of formulas
  $\{\varphi(x,a_\sigma)| \ \sigma\in\lambda^{<\kappa}\}$ witnessing  SOP$_1$ w.r.t. $\lambda^{<\kappa}$\, (i.e. for each $\beta\in \lambda^\kappa$,  $q_{\beta}(x): = \{ \varphi (x, a_{\beta\lceil i})\mid i<\kappa \}$ is consistent,  while for any $\alpha\in \lambda^{<\kappa}$ and $u<v\in\lambda$, 
  $\{ \varphi(x,a_\gamma),\varphi(x,a_{\alpha^\smallfrown\la v\ra})\}$ is inconsistent for any  $\gamma \unrhd \alpha^\smallfrown\la u\ra$). Let $A$ be the set of
parameters in the tree.
We let $\CF: = \{ q_{\beta}(x)\mid \beta \in \lambda^\kappa \}$. Note that 
$|\CF|=\lambda^\kappa>
\lambda=\lambda^{<\kappa}=|A|.$

\medskip

We want to show that $\CF$ satisfies the condition (2) in  Definition \ref{nt1}, Thus assume a set $\CG=\{p_\beta\subseteq q_\beta\mid \beta\in\tau \}$ is given where $|q_\beta\smallsetminus p_\beta|<\kappa$ and $\tau\subseteq \lambda^\kappa$ with $|\tau|=\lambda^+$.  Since $|q_\beta\smallsetminus p_\beta|<\kappa$ and $\kappa$ is regular, for each $\beta\in \tau$, there must exist an ordinal $i_\beta<\kappa$ such that $\{\varphi(x,a_{\beta\lceil i}) \mid i_\beta\leq i<\kappa\} \subseteq p_\beta$. Note that  $\lambda^{<\kappa}=\lambda$ implies $\kappa < \lambda^+$. Thus  there exists a subset $\tau'' \subseteq \tau$ 
of size $\lambda^+$ such that $i_\beta$ stays the same, say $i_0$ for every $\beta\in \tau''$. Once more, since $\lambda^{<\kappa}=\lambda$,  for some subset $\tau' \subseteq \tau''$ 
of size $\lambda^+$, $\beta\lceil i_0$ stays the same for every $\beta\in\tau'$. Namely, there is 
$\sigma_0\in \lambda^{<\kappa}$ such that $\sigma_0=\beta\lceil i_0$ (and hence $a_{\sigma_0}=a_{\beta\lceil i_0}$) for all 
$\beta\in \tau'$.

Now let $\CG':=\{ p_\beta\in \CG\mid \beta \in \tau'\}$, and  for $\sigma(\unrhd \sigma_0)\in \lambda^{<\kappa}$, we let
$\CG'_\sigma:=\{p\in \CG' \mid \varphi(x,a_{\sigma})\in p\}$.

\medskip

\noindent {\em Claim.}  There are $\mu(\unrhd \sigma_0)\in \lambda^{<\kappa}$ and $s_0< s_1\in \lambda$ such that
$|\CG'_{\mu^\smallfrown \la s_0\ra}|=|\CG'_{\mu^\smallfrown\la s_1\ra}|=\lambda^+$:         Suppose  not. Thus for each $\sigma\unrhd \sigma_0\in \lambda^{<\kappa}$ there is at most one $s<\lambda$ such that  
$|\CG'_{\sigma^\smallfrown\la s\ra}|=\lambda^+$. Then it lead a contradiction  by the similar cardinality computation in the proof 
of Claim in that of Theorem \ref{sopt4} (1)$\Rightarrow$(6). Hence the claim follows.

\medskip

Now let   $\tau_0,\tau_1$ be
the disjoint  subsets of $\tau'$ indexing  the sets $\CG'_{\mu^\smallfrown \la s_0\ra}$ and $\CG'_{\mu^\smallfrown\la s_1\ra}$, respectively,
so $\CG'_{\mu^\smallfrown \la s_j\ra}=\{p_i\in \CG'\mid i\in \tau_j\}$ ($j=0,1$).
We now put for each $i\in \tau_0\cup \tau_1$, $p'_i:=p_i\smallsetminus q_\mu$ where 
$q_\mu=\{\psi(x,a_\sigma)| \ \sigma \trianglelefteq \mu \}$.

Notice that the formula $\varphi(x,a_{\mu s_1})\in \bigcap_{i\in\tau_1} p'_i$ is inconsistent with any formula 
in $\bigcup_{i\in \tau_0} p'_i$.
Hence Definition \ref{nt1}(2) holds.

\medskip

(6)$\Rightarrow$(5)$\Rightarrow$(2)$\Rightarrow$(4)
and (6)$\Rightarrow$(3)$\Rightarrow$(2) Clear.

\medskip

(4)$\Rightarrow$(1) Assume  $\NT_\varphi^1(\kappa,\lambda)\geq \lambda^+$ for some infinite $\lambda$ and $\kappa$. Hence there is a family $\CF=\{q_i\mid i<\lambda^+\}$ over
 a set $A$ with $|A|\leq \lambda$ satisfying the conditions  Definition \ref{nt1}(1) and (2). 
We will produce an SOP$_1$ tree for $\varphi$ from $\CF$.

\medskip

\noindent {\em Claim.} There exist a  function $f: 2^{<\omega}\to A$, a family $\{\, \CG_\sigma\mid \sigma\in  2^{<\omega}\, \}$ of families of types, and a family 
$\, \{\, \tau_\sigma \subseteq
 \lambda^+\mid \sigma\in  2^{<\omega} \}\, $ such that, for all  $\sigma\in  2^{<\omega}$, 

\be
\item[(i)] $|\tau_{\sigma}|=\lambda^+$;  $\tau_{\sigma 0}$ and $\tau_{\sigma 1 }$ are disjoint subsets of $\tau_{\sigma}$,

\item[(ii)]  $\CG_\sigma $ is of the form $\{p_i\mid p_i\subseteq q_i,\  i\in \tau_\sigma\}$ (so $|
\CG_\sigma|=\lambda^+$) with
$|q_i\smallsetminus p_i|<\kappa$, and for $j\in \{0,1\}$, $\CG_{\sigma j}$ is of the form $\{p'_i\mid p'_i\subseteq p_i\in 
\CG_\sigma, \  i\in \tau_{\sigma j}\}$ with $|p_i\smallsetminus p'_i|<\kappa$,

\item[(iii)] for $a_\sigma:=f(\sigma)$ we have    $\varphi(x, a_\sigma)\in \bigcap \CG_\sigma$,  
and  $\varphi(x, a_{\sigma1})\in \bigcap\CG_{\sigma 1}$  is inconsistent with every formula in $\bigcup\CG_{\sigma0}$.
\ee

\medskip

\noindent {\em Proof of Claim.}
We construct such a function and sets by induction on the length of $\sigma$. When $\sigma=\emptyset$, choose $\varphi(x,b_i)$ from each $q_i\in\CF$. Then 
since $|A| \leq  \lambda$, there is  a subset
$\tau_{\emptyset}\subseteq \lambda^+$ of size $\lambda^+$ such that
the $b_i$ are equal (say, to $a_{\emptyset}$) for all $i\in \tau_{\emptyset}$. Then set $f(\emptyset)=a_\emptyset$. Also, set $\CG_{\emptyset}:=  \{q_i \mid i\in \tau_{\emptyset}\}$, so $\varphi(x,a_\emptyset)\in \bigcap \CG_\emptyset$.

\medskip
Assume now the induction hypothesis for $\sigma$. We will find
sets and  function values corresponding to
$\sigma 0$ and $\sigma 1$. Write  $\CG_\sigma=\{p_i\mid i\in \tau_\sigma\}$.
Since $\CF$ satisfies Definition \ref{nt1}(2), there exist disjoint subsets $\tau'_{\sigma j } \subseteq \tau_{\sigma}\, $ of size $\lambda^+$ ($j=0,1$) and a subset $p'_i\subseteq p_i$ with $|p_i \smallsetminus p'_i|<\kappa$ for each $i\in \tau'_{\sigma 0 }\cup \tau'_{\sigma 1 }$, such that for every $p'_i\in \CH_1$, there is a formula $\varphi(x,a'_i)\in p'_i$ inconsistent with each formula in $\bigcup \CH_0$, where
 $\CH_j=\{ p'_i \mid i\in \tau'_{\sigma j}\}$. 
 
 Now since again $|A|\leq \lambda$,  
there must be a set $\tau_{\sigma 1}\subseteq \tau'_{\sigma 1}$ with $|\tau_{\sigma 1}|=\lambda^+$  such that   $a'_i$ are all equal for all $i\in \tau_{\sigma 1}$, which we put $f(\sigma 1)=a_{\sigma 1}$. Thus if we let 
  $\CG_{\sigma 1}:=\{ p'_i\mid i\in \tau_{\sigma 1}\}$, then $\varphi(x,a_{\sigma 1})\in \bigcap \CG_{\sigma1}$ is inconsistent with each formula in $\bigcup \CH_0$. Similarly if we choose $\varphi(x, b'_i)\in q'_i\in \CH_0$, there must be a subset $\tau_{\sigma 0}\subseteq \tau'_{\sigma 0}$ of size $\lambda^+$ such that $b'_i$ stays the same for each $i\in \tau_{\sigma 0}$, which  we let $f(\sigma 0)= a_{\sigma 0}$.  Then let  $\CG_{\sigma 0}:=\{p'_i\mid i\in \tau_{\sigma 0}\}$, so $\varphi(x, a_{\sigma 0})\in \bigcap\CG_{\sigma 0}$. Therefore,  $\tau_{\sigma j }$, $f(\sigma j)$ and $\CG_{\sigma j}$, for $j=0, 1$, satisfy all the required conditions for the induction step, and the proof for Claim is complete.
\medskip

Now, using the properties  described in Claim, we see that  the tree $ \{\varphi(x,a_{\sigma})\mid  \sigma\in 2^{<\omega}\}$ witnesses SOP$_1$.
 Indeed given any $\sigma \in 2^{<\omega}$,
the formula $\varphi(x,a_{\sigma 1})$ is inconsistent with any $\varphi(x,a_\gamma)$ where $\gamma\unrhd \sigma 0$.
\end{proof}

We finish this section by asking the following: Given a theory, are there criteria for SOP$_1$ analogous to Theorem \ref{sopt5} for SOP$_2$?

\section{Kim-forking and TP$_2$}

We begin this section by recalling basic definitions. 

\begin{definition}\label{kimfking}
\be
\item We say  a formula $\varphi(x,a_0)$ {\em divides over}  a set $A$, if there is an $A$-indiscernible sequence $\la a_i\mid i<\omega\ra$
such that $\{\varphi(x,a_i)\mid i<\omega\}$ is inconsistent.  A formula {\em forks} over $A$ if the formula implies a finite disjunction of formulas, each of which divides over $A$.  A type {\em divides/forks} over $A$ if the type implies a formula which divides/forks over $A$. We write $a\ind_AB$  ($a\ind^d_A B$) if $\tp(a/AB)$ {\em does not fork } ({\em divide}, resp.) over $A$.

\item An $A$-indiscernible sequence $\la a_i\mid i<\omega\ra$ is said to be a {\em Morley sequence over} $A$ if
$a_i\ind_A a_{<i}$ holds for each $i<\omega$.  

\item We say a formula $\varphi(x,a_0)$ {\em Kim-divides over} $A$ if $\{\varphi(x,a_i)\mid i<\omega\}$ is inconsistent for some Morley sequence $\la a_i\mid i<\omega \ra$ over $A$.  A formula {\em Kim-forks} over $A$ if the formula implies a finite disjunction of formulas, each of which Kim-divides over $A$. 
\item A type {\em Kim-divides/forks} over $A$ if the type implies a formula which Kim-divides/forks over $A$. We write $c\ind^K_AB$ if $\tp(c/AB)$ {\em does not Kim-fork } over $A$. Hence $\ind\Rightarrow \ind^K\mbox{ and } \ind^d$.
\ee
\end{definition}

Originally in \cite{KR}, the notion of {\em Kim-dividing} is introduced {\em over a model}, using the notion of a Morley sequence in a global invariant extension of a type over the model. There it is shown that, over a model, that notion is equivalent to the one stated in Definition \ref{kimfking}(3).   Since  in general even in a simple theory, there need not exist a global invariant extension of a type over a set,  instead in \cite{DKR}  the above definition in (3) is coherently given  as {\em Kim-dividing over an arbitrary set}.

As is well-known, in any simple $T$, $\ind$ satisfies symmetry, full transitivity (that is: for any $d$ and $A\subseteq B\subseteq C$, $d\ind_AB$ and $d\ind_BC$ iff $d\ind_AC$), extension, local character, finite character, and
$3$-amalgamation of Lascar types.  Moreover  in such $T$, $\ind=\ind^d=\ind^K$ \cite{K}, and nonforking existence (that is: $d\ind_AA$ for any $d$ and $A$) holds.
As we will not deal with these facts, see \cite{C1} or  \cite{K1}  for more details. Further advances are discovered in \cite{KR},\cite{KR1},\cite{DKR},\cite{CKR} recently.  Namely, it is shown that  in any NSOP$_1$ $T$ having nonforking existence (as said any simple $T$, and all the known NSOP$_1$ $T$ have this),    the notions of Kim-forking and Kim-dividing coincide, and $\ind^K$ supplies a good independence notion since it satisfies all the aforementioned properties that  hold of $\ind$  in simple theories, except  base monotonicity (so  there can exist $d$ and $A\subseteq B\subseteq C$ such that $d\ind^K_AC$ but $d\dep^K_BC$ holds). 
 
 \medskip 
 
 In this section we study TP$_2$ in relation with Kim-forking. In particular we show that 
 if $T$ has TP$_2$ then there is a non-continuous Kim-forking chain of arbitrarily large length (Proposition \ref{kfkseq}),   
 by which we prove that $T$ is supersimple iff there is no Kim-forking chain of length $\omega$ (Theorem \ref{ssimple}).  
We also show that in any $T$ holding TP$_2$,  there is a type having  arbitrarily large local weight 
with respect to $\ind^K$ (Proposition \ref{lcwt}).

This section might be considered as an expository note, since all the results in this section are more or less straightforward  consequences of known facts \ref{nsop1cblelc} and \ref{tp2}. In particular,  
the referee of this paper points out to us that Proposition  \ref{kfkseq} follows from a result in \cite{Ch}.

\medskip 
Recall that a sequence $\la A_i\mid i<\kappa\ra$ of sets is said to be 
 {\em continuous} if 
for each limit $\delta<\kappa$, 
  $A_\delta=\bigcup_{i<\delta} A_i$.

 \begin{Fact} \cite{KRS} \label{nsop1lc}
 The following are equivalent.
 \be\item
 $T$ is NSOP$_1$.
 
 \item
  There do not exist finite $d$ and a continuous  increasing  sequence  $\la M_i\mid i<|T|^+\ra $ of $|T|$-sized models 
  such that 
  for each  $i<|T|^+$,   $d\depk_{M_{i}}M_{i+1}$.

\ee
\end{Fact}

Indeed the following is implicitly shown in \cite{KRS} using Fact \ref{cr}.  We supply a proof for completeness.

\begin{Fact}\label{nsop1cblelc}
If $T$ has SOP$_1$ then for each infinite cardinal $\kappa$, there exist a finite tuple  $d$ and a continuous  increasing  sequence  $\la A_\alpha \mid \alpha <\kappa\ra$ of sets 
  such that 
  for each  $\alpha <\kappa$,   $|A_\alpha |\leq |\alpha| \cdot\omega$ and $d\depk_{A_{\alpha}}A_{\alpha+1}$.
\end{Fact}
\begin{proof} Assume $T$ has SOP$_1$.
Given an infinite $\kappa$, by using compactness, there are a formula $\varphi(x,y)$ and  an indiscernible 
 sequence $\la a_ib_i\mid i\in \mathbb{Z}\cdot \kappa\ra$ satisfying Fact \ref{cr}.  Namely, $a_i\equiv_{(ab)_{<i}} b_i$ for all $i\in\mathbb{Z}\cdot \kappa$, 
$\{\varphi(x,a_i)\mid i\in \mathbb{Z}\cdot \kappa\}$ is realized by say $d$,  and 
$\{\varphi(x,b_i)\mid i\in\mathbb{Z}\cdot \kappa\}$ is $2$-inconsistent  \ (*). 

Now for $n <\omega $, let $A_n=\{a_ib_i\mid i \in \mathbb{Z}\cdot (n+1)\}$, and  for $\omega\leq \alpha <\kappa$, let $A_\alpha=\{a_ib_i\mid i \in \mathbb{Z}\cdot \alpha\}$. Then clearly $\la A_\alpha \mid \alpha\in \kappa\ra$ is a continuous  increasing sequence with $|A_\alpha|\leq \omega\cdot |\alpha|$.   Moreover 
for 
each $b_i\in A_{\alpha+1}\setminus A_{\alpha}$,  the  countable sequence  $I_{b_i}(\subset  A_{\alpha+1}\setminus A_{\alpha})$ of successive $b_j$'s starting from $b_i$
is  a finitely satisfiable indiscernible  (so  Morley) sequence in $\tp(b_i/A_\alpha)$. Thus by (*), 
$\varphi(x,b_i)$ Kim-divides over $A_\alpha$. Then since  $a_i\equiv_{A_\alpha} b_i$, again by (*)
we have that $d\depk_{A_{\alpha}}a_i$. Note that  $a_i\in A_{\alpha+1}\setminus A_{\alpha}$.
Hence  $d\depk_{A_{\alpha}}A_{\alpha+1}$ as wanted.
\end{proof}

Contrary to Fact \ref{nsop1lc}(2), as in the following example,   in NSOP$_1$ $T$, there can exist  a {\em{non-continuous}} increasing Kim-forking sequence of length $|T|^+$ of $\leq |T|$-sized sets, and  arbitrary lengths {\em continuous} increasing Kim-forking sequences. 

\begin{Example}\label{paraequi} \cite{CKR}
 Let $T$ be the theory of 
  the  random parametrized equivalence relations, i.e.,  the  the Fra\"{i}ss\'{e}
limit of the class of finite models with two sorts $(P,E)$ and a ternary relation
$\, \sim\, $ on $\, P\times P\times E\, $ such that, for each $e\in E$,
$x\sim_ey$ forms an equivalence relation on $P$.
 
 So in a model of $T$,  there are two sorts $P$ and $E$ as described above.
  Let $d\in P$. Given a cardinal $\kappa$, choose distinct
 $e_i \in E$, and $d_i\in P$ ($i<\kappa)$ such that $d\sim_{e_i}d_i$, but $d_j\not\sim_{e_k} d_i$ for each  $j< i$ and 
 $k\leq i$. 
 Let $D_i=((ed)_{ <i})e_i$.  Note that the sequence $\la D_i\mid i<\kappa\ra $ is increasing but {\em not continuous} (for example,  $ D_{<\omega} \subsetneq D_{\omega}$). Notice further that $d\depk_{D_i}d_i$, so   $d\depk_{D_i}D_{i+1}$ for each $i<\kappa$. 
 
Moreover,  there is a  continuous increasing  Kim-forking sequence of length 
$\kappa$ of $\kappa$-sets.  We work with the same chosen elements above. Let $C:=\{e_i\mid i<\kappa\}\subset E$, and let
  $C_i:=Cd_{<i}$. Clearly $\la C_i\mid i<\kappa\ra$ is a continuous  increasing sequence of $\kappa$-sets.  
 Now for each $i<\kappa$, it follows $d\depk_{C_i}d_i$, and hence     $d\depk_{C_i}C_{i+1}$. 
\end{Example}

Now we can ask whether such phenomena happen in any non-simple NSOP$_1$ $T$.  We show that indeed in any theory with TP$_2$, such sequences can be found. 
The following fact is well-known and a proof can
 be found for example in \cite{KKS}.  Recall that an array $\la a_{ij}\mid i<\kappa, j<\lambda \ra$ is said to be {\em indiscernible}\footnote{In some literature this notion is called {\em strongly indiscernible}} over $A$ if for 
 $L_i:=\la a_{ij}\mid j<\lambda \ra$,  $\la L_i\mid i<\kappa\ra$ is $A$-indiscernible, and $A$-mutually indiscernible (i.e., 
  each $L_i$ is indiscernible over $\bigcup \{L_j\mid j(\ne i)<\kappa\}A$).

\begin{Fact}\label{tp2} The following are equivalent.
\be\item
$\varphi(x,y)$ has TP$_2$.

\item Let $\kappa$ be an infinite cardinal. There is an indiscernible  array $\la a_{ij}\mid i<\kappa, j<\omega+\omega \ra$ such that 
\be\item for each $i<\kappa$, $\{\varphi(x,a_{ij})\mid j<\omega+\omega\}$ is $2$-inconsistent,  and

\item for any $f:\kappa \to \omega+\omega$, $\{\varphi(x, a_{if(i)})\mid i<\kappa\}$ is consistent. 

\ee
\ee
\end{Fact}

\begin{Proposition}\label{kfkseq}
 Assume $T$ has TP$_2$. Let $\kappa$ be an infinite cardinal.
\be\item
There are a finite tuple $d$ and  an increasing {\em non-continuous} sequence of sets $A_i$ $(i<\kappa)$ of size $|i|\cdot\omega(<\kappa)$ such that $d\dep^K_{A_i}A_{i+1}$ for each $i<\kappa$. In particular there is an increasing countable  sequence of countable sets $B_i$ such that $d\dep^K_{B_i}B_{i+1}$ for each $i<\omega$. 

\item  There are a finite tuple $d$ and  an increasing {\em continuous} sequence of sets $E_i$ $(i<\kappa)$ of size $\kappa$ such that $d\dep^K_{E_i}E_{i+1}$ for each $i<\kappa$.
\ee
\end{Proposition}
\begin{proof}
(1)  Due to Fact \ref{tp2} and compactness, there are a formula $\varphi(x,y)$ and  an array $\la a_{ij}\mid i<\kappa,\  j\in \omega+\omega^* \ra$ 
 where $\omega^*:=\{i^*\mid i\in \omega\}$ with the reversed order of $\omega$ (so for $i^*, j^*\in \omega^*$, we have $n<i^*$ for all $n\in\omega$,  and $j^*<i^*$ if $i<j$) such that 
\be\item[(a)] for each $i<\kappa$, $\{\varphi(x,a_{ij})\mid j\in \omega+\omega^*\}$ is $2$-inconsistent,  

\item[(b)] for any $f:\kappa \to \omega+\omega^*$, $\{\varphi(x, a_{if(i)})\mid i<\kappa\}$ is consistent, and

\item[(c)] the array is mutually indiscernible, i.e.,  for any $i<\kappa$, 
$L_i:=\la a_{ij}\mid j\in \omega+\omega^*\ra$ is indiscernible over $\bigcup \{L_j\mid j(\ne i)<\kappa\}$. 
\ee

 For each $i\in\kappa$, we let $I_i:=\la a_{ij}\mid j<\omega\ra$, and let $J_i:=\la a^*_{ij}\mid  j<\omega\ra$ where $a^*_{ij}=a_{ij^*}$ with $j^*\in\omega^*$,  so as a set  $L_i=I_i\cup J_i$. 

Now due to (b), there is  $d\models \{\varphi(x, a^*_{i 0})\mid i<\kappa\}$. Put
$A_i=\{I_k\mid k\leq i\}\cup \{a^*_{k0}\mid k<i\}$.  Then $|A_i|=|i|\cdot\omega$. Now  by (c), $J_i$ is finitely satisfiable, so Morley over $A_i$. Hence, by (a) we have 
$$d\dep^K_{A_i}a^*_{i0} \mbox{ \ and \ } d\dep^K_{A_i}A_{i+1}$$
for each $i<\kappa$. Notice that the sequence $\la A_i\mid i<\kappa\ra$ is not continuous, for example $A_{<\omega}= A_\omega\setminus I_\omega\subsetneq A_\omega$.

For the second statement of (1), simply put $B_i=A_i$ for $i<\omega$. 

\medskip

(2) We keep use the same $d$ in (1).  Let $E:=I_{<\kappa}$, and for $i<\kappa$ let
$E_i:= E\cup\{a^*_{k0}\mid k<i\}$.  Now due to (c) again, for $i\in\kappa$, 
$J_i$ is Morley over $E_i$. 
Therefore we have
$$d\dep^K_{E_i}a^*_{i0} \mbox{ \ and hence  \ } d\dep^K_{E_i}E_{i+1}.$$
as wanted. Note that clearly $\la E_i\mid i<\kappa\ra$ is a continuous increasing sequence with each $|E_i|=\kappa$. 
\end{proof}

As said before Fact \ref{nsop1lc}, the referee of this paper points out that Proposition  \ref{kfkseq} directly follows from 
the proof of Lemma 4.7 in \cite{Ch} as well. Thus the above proof might be considered as the one describing  
  the proposition as  a straightforward consequence  of Fact \ref{tp2}.    

\medskip

Now we recall that $T$ is {\em supersimple} if for any finite $a$, and  a set $A$, there is a finite subset $A_0$ of $A$ such that
$a\ind_{A_0}A$.  As is well-known $T$ is supersimple iff there does not exist a countably infinite forking chain (see for example, \cite{K1}). The following theorem says that the same holds with a countably infinite Kim-forking chain. 

\begin{Theorem}  \label{ssimple} The following are equivalent.
\be\item 
$T$ is supersimple.

\item  There do not exist finite $d$ and an increasing sequence
of  sets $A_i$ ($i<\omega$) such that $d\dep^K_{A_i} A_{i+1}$ for each $i<\omega$. 

\item There do not exist finite $d$ and an increasing sequence
of  countable sets $A_i$ ($i<\omega$) such that $d\dep^K_{A_i} A_{i+1}$ for each $i<\omega$. 

\ee
\end{Theorem}
\begin{proof} (1)$\Rightarrow$(2) is well-known as said before this theorem, and (2)$\Rightarrow$(3) is obvious.
\medskip 

(3)$\Rightarrow$(1)\  We prove this  contrapositively. Suppose $T$ is not supersimple. If $T$ is simple, then since $\ind=\ind^K$, 
again it is well-known that there exist such a tuple and a sequence described in (3). If $T$ is NSOP$_1$ but not simple, then $T$ has TP$_2$ and Proposition \ref{kfkseq} says there are such a tuple and a sequence.  
If $T$ has SOP$_1$ then the existence of such a tuple and a sequence is guaranteed in Fact \ref{nsop1cblelc}.
\end{proof}

As pointed out in \cite{KR1},   $T$ is  NSOP$_1$ iff   there do not exist 
 tuples  $a_i$  ($i<\omega$), a model   $M$, and an $\CL$-formula $\varphi(x,y)$ 
such that for each $i<\omega$, 
  $a_i\equiv_M a_0$,  $a_i\indk_M a_{<i}$,   
  $\varphi(x,a_i)$ Kim-divides over $M$, and $\{\varphi(x, a_i)\mid i<\omega\}$ is consistent.  However only a slightly  weaker condition always holds in any $T$ having TP$_2$.  

\begin{Proposition}\label{lcwt}
Assume $\varphi(x,y)$ has TP$_2$. Then for each infinite $\kappa$, there are a set  $A$ with $|A|\leq\kappa$, and finite tuples $d$, $c_i$  ($i<\kappa$) 
such that 
\be
\item $d\models \varphi(x, c_i)$, 

\item   $\varphi(x,c_i)$ Kim-divides over $A$ (so 
 $d\depk_A c_i$) witnessed by a Morley sequence $(c_i\in)J_i$ over $A$ 
 with $J_i\equiv J_0$ (so $c_i\equiv c_0$), and

\item  $ c_i\ind_A\{c_{k}\mid k <\kappa, \ k\ne i\}$ (so  $ c_i\indk_A\{c_{k}\mid k<\kappa, \ k\ne i\}$).
\ee 
\end{Proposition}
\begin{proof} 
As in the proof of Proposition \ref{kfkseq},  there is  an indiscernible  array $\la a_{ij}\mid i<\kappa,\  j\in \omega+\omega^* \ra$ 
 such that 
\be\item[(a)] for each $i<\kappa$, $\{\varphi(x,a_{ij})\mid j\in \omega+\omega^*\}$ is $2$-inconsistent,  

\item[(b)] for any $f:\kappa \to \omega+\omega^*$, $\{\varphi(x, a_{if(i)})\mid i<\kappa\}$ is consistent, and

\item[(c)]   for any $i<\kappa$, 
$L_i=\la a_{ij}\mid j\in \omega+\omega^*\ra$ is indiscernible over $\bigcup \{L_j\mid j<\kappa, \  j\ne i\}$. 
\ee

Again for each $i\in\kappa$, let $I_i=\la a_{ij}\mid j<\omega\}$,  and   $J_i=\la a^*_{ij}\mid  j<\omega\ra$ where $a^*_{ij}=a_{ij^*}$ with $j^*\in\omega^*$. We further let $c_i:= a^*_{i 0}$. Now by (b), there is 
$d\models \{\varphi(x, c_i)\mid i<\kappa\}$. 

We now put $A:=I_{<\kappa}$, so $|A|=\kappa$. Now due to (c), each $J_i$ is a Morley sequence over  $A$ , and 
$\tp(c_i/A\{c_{k}\mid k <\kappa, \ k\ne i\})$ is finitely satisfiable in $A$. Hence (3) follows, and (2) follows as well due to (a).
\end{proof}


\begin{thebibliography}{99}

\bibitem{C} Enrique Casanovas.
\newblock The number of types in simple theories. 
\newblock {\it Annals of Pure and Applied Logic} 98 (1999), 69-86.

\bibitem{C1} Enrique Casanovas.
\newblock  {\it Simple theories and hyperimaginaries}.
\newblock Cambridge University Press, Cambridge, 2011.

\bibitem{Ch} Artem Chernikov.
\newblock  Theories without the tree property of the second kind. 
\newblock {\it Annals of Pure and Applied Logic} 165  (2014), 695-723.


\bibitem{CKR} Artem Chernikov, Byunghan Kim, and Nicholas Ramsey.
\newblock Transitivity, lowness, and ranks in NSOP$_1$ theories.
\newblock  Preprint. 

\bibitem{CR} Artem Chernikov and Nicholas Ramsey.
\newblock On model-theoretic tree properties.
\newblock  {\it Journal of Mathematical  Logic} 16 (2016). 

\bibitem{DKR} Jan Dobrowolski, Byunghan Kim, and Nicholas Ramsey.
\newblock Independence over arbitrary sets in NSOP$_1$ theories. 
\newblock  Preprint. 


 
\bibitem{KR} Itay Kaplan and Nicholas Ramsey.
\newblock On Kim-independence. 
\newblock To appear in {\it Journal of 
European Mathematical Society}.

 
\bibitem{KR1} Itay Kaplan and Nicholas Ramsey.
\newblock Transitivity of  Kim-independence. 
\newblock https://arxiv.org/abs/1901.07026.


\bibitem{KRS}  Itay Kaplan, Nicholas Ramsey, and Saharon Shelah.
\newblock Local character of Kim-independence.
\newblock  {\it Proceedings of American Mathematical Society} 147 (2019), 1719-1732.

\bibitem{K} Byunghan Kim.
\newblock  Forking in simple unstable theories.
\newblock {\it Journal of London Mathematical Society} 57 (1998), 257-267. 

\bibitem{K1} Byunghan Kim.
\newblock {\it Simplicity theory}.
\newblock Oxford University Press, Oxford, 2014.

\bibitem{KK} Byunghan  Kim and Hyeung-Joon Kim.
\newblock Notions around tree property $1$.
\newblock  {\it Annals of Pure and Applied Logic}  162 (2011), 698-709. 

\bibitem{KKS} Byunghan  Kim, Hyeung-Joon Kim,  and Lynn Scow.  
\newblock Tree properties, revisited.
\newblock {\it Archive for Mathematical Logic} 53 (2014), 211-232.

\bibitem{KP} Byunghan Kim and Anand Pillay.
\newblock Simple theories. 
\newblock  {\it Annals of Pure and Applied Logic}  88 (1997), 149-164. 

\bibitem{MS} Maryanthe Malliaris and Saharon Shelah.
\newblock Model-theoretic applications of cofinality spectrum problems.
\newblock {\it Israel Journal of Mathematics} 220 (2017), 947-1014.

\bibitem{S} Saharon Shelah.
\newblock {\it Classification theory and the number of non-isomorphic models}, revised. 
\newblock North-Holland, Amsterdam, 1990. 

\bibitem{S0} Saharon Shelah.
\newblock Simple unstable theories.
\newblock {\it Annals of Mathematical Logic} 19 (1980), 177-203.

\bibitem{S1} Saharon Shelah.
\newblock Toward classifying unstable theories.
\newblock {\it Annals of Pure and Applied Logic} 80 (1996), 229-255.

\end{thebibliography}
\end{document}